\documentclass[12pt]{amsart}
\usepackage{dirtytalk}
\usepackage[english]{babel}
\usepackage[utf8x]{inputenc}
\usepackage[T1]{fontenc}
\usepackage{amsmath}
\usepackage{amstext}
\usepackage{amsfonts}
\usepackage{amssymb}
\usepackage{amsthm}
\usepackage[alphabetic, initials]{amsrefs}
\usepackage{mathtools}
\usepackage{dsfont}
\usepackage{color}
\usepackage{graphicx}
\usepackage[colorinlistoftodos]{todonotes}
\usepackage[colorlinks, linkcolor=red, citecolor=blue, urlcolor=blue, hypertexnames=true]{hyperref}
\usepackage{float}
\usepackage{theoremref}
\usepackage{enumitem}
\allowdisplaybreaks

\newtheorem{theorem}{Theorem}[section]
\newtheorem{lemma}[theorem]{Lemma}
\newtheorem{prop}[theorem]{Proposition}
\newtheorem{cor}[theorem]{Corollary}

\newtheorem*{cor*}{Corollary}
\newtheorem*{conjecture*}{Conjecture}
\newtheorem*{thm*}{Theorem}
\newtheorem*{lem*}{Lemma}
\newtheorem*{prop*}{Proposition}
\theoremstyle{definition}
\newtheorem{definition}[theorem]{Definition}
\newtheorem{defn}[theorem]{Definition}
\newtheorem{example}[theorem]{Example}
\newtheorem*{defn*}{Definition}
\theoremstyle{remark}
\newtheorem{remark}[theorem]{Remark}

\newcommand{\cA}{\mathcal{A}}   
\newcommand{\cM}{\mathcal{M}}   
\newcommand{\bC}{{\mathbb{C}}}
\newcommand{\bR}{{\mathbb{R}}}
\newcommand{\E}{{\mathbb{E}}}

\newcommand{\bN}{{\mathbb{N}}}
\newcommand{\bZ}{{\mathbb{Z}}}

\title[Relative ISR-property]{On Relative Invariant Subalgebra Rigidity Property}
\author[Amrutam]{Tattwamasi Amrutam}
\address{Institute of Mathematics of the Polish Academy of Sciences, Ul.~\'Sniadeckich 8, 00-656 Warszawa, Poland}
\email{tattwamasiamrutam@gmail.com}
\thanks{The author was partially supported by the Simons Foundation grant (award no. SFI-MPS-T-Institutes-00010825) and from State Treasury funds as part of a task commissioned by the Minister of Science and Higher Education under the project \say{Organization of the Simons Semesters at the Banach Center - New Energies in 2026-2028} (agreement no. MNiSW/2025/DAP/491).}
\date{\today}

\begin{document}
\begin{abstract}
A countable discrete group \(\Gamma\) is said to have the \emph{relative ISR-property} if for every non-trivial normal subgroup
\(N\trianglelefteq\Gamma\) and every von Neumann subalgebra \(\cM\subseteq L(\Gamma)\)
invariant under conjugation by \(N\), one has \(\cM=L(K)\) for some subgroup
\(K\le\Gamma\). Similarly, \(\Gamma\) has the \emph{relative $C^*$-ISR-property} if every 
\(N\)-invariant unital \(C^*\)-subalgebra \(\cA \subseteq C_r^*(\Gamma)\) is of the form \(C_r^*(K)\). 
We show that every torsion-free acylindrically hyperbolic group with trivial amenable
radical satisfies the relative ISR property. Moreover, we also show that all torsion-free hyperbolic groups have the relative $C^*$-ISR property. Furthermore, we establish an analogous relative ISR-property for irreducible lattices in higher-rank semisimple Lie groups, such as \(\mathrm{SL}_d(\mathbb{Z})\) (\(d \geq 3\)), with trivial center.
\end{abstract}
\maketitle

\section{Introduction}
The structural theory of von Neumann algebras associated with discrete groups has witnessed
a remarkable resurgence through its deep interplay with geometric group theory and Popa's deformation/rigidity paradigm. A central
theme in this area is the rigidity of subalgebras \(\cA\subseteq L(\Gamma)\) that remain
invariant under conjugation by $\Gamma$.

Motivated by the works of Chifan--Das~\cite{chifan2020rigidity} and
Alekseev--Brugger~\cite{alekseev2021rigidity} on \(\Gamma\)-invariant subalgebras for
negatively curved groups and higher-rank lattices, Kalantar--Panagopoulos~\cite{kalantar2023invariant}
proved that for irreducible lattices \(\Gamma\) in higher-rank semisimple Lie groups,
every \(\Gamma\)-invariant von Neumann subalgebra of \(L(\Gamma)\) arises as the group
von Neumann algebra of some normal subgroup of \(\Gamma\).

This result strongly motivated the author and Jiang~\cite{amrutam2023invariant} to initiate a
systematic study for countable groups. In our work, we introduced the \emph{invariant
subalgebra rigidity (ISR) property} (a group \(\Gamma\) is said to have the ISR property
if every \(\Gamma\)-invariant von Neumann subalgebra of \(L(\Gamma)\) is of the form
\(L(K)\) for some normal subgroup \(K\trianglelefteq\Gamma\)). The ISR property has proven extremely
fruitful and has been generalized and strengthened in many subsequent works, including
those of Chifan--Das--Sun~\cite{chifan2023invariant}, Dudko--Jiang~\cite{dudko2024character},
Jiang--Zhou~\cite{jiang2024example},
Amrutam--Dudko--Jiang--Skalski~\cite{amrutam2025invariant}. More recently, Jiang, along with Li and Liu~\cite{JL, jiang2026classification}, has taken a new direction, systematically studying invariant subalgebras in cases lacking the ISR property.

In recent years, the rigidity of invariant subalgebras in both the von Neumann algebraic and the reduced \(C^*\)-algebraic settings has emerged as a powerful tool that bridges operator algebras with geometric group theory.
The present paper addresses a natural yet subtler variant of this question: rigidity when
invariance is required only under a non-trivial normal subgroup
\(N\trianglelefteq\Gamma\) rather than under the whole group. Our goal is to develop a rigidity theory in two distinct but robust geometric settings: the large and geometrically flexible class of \emph{torsion-free hyperbolic groups}, and the rigid realm of irreducible lattices in higher-rank semisimple Lie groups, such as $\mathrm{SL}_d(\mathbb{Z})$ ($d \geq 3$). To motivate our results, we introduce the following relative version.

\begin{definition}
A group $\Gamma$ is said to have the \emph{relative ISR-property} if for every non-trivial
normal subgroup $N \trianglelefteq \Gamma$ and every von Neumann subalgebra
$\cM \subseteq L(\Gamma)$ invariant under conjugation by $N$, one has $\cM = L(K)$ for
some subgroup $K \leq \Gamma$. Moreover, we say that $\Gamma$ has \emph{relative $C^*$-ISR-property} if every $N$-invariant unital $C^*$-subalgebra $\mathcal{A}\le C_r^*(\Gamma)$ is of the form $C_r^*(K)$ for some subgroup $K\le \Gamma$.
\end{definition}

Our first main result shows that this property holds for the class of acylindrically hyperbolic groups.

\begin{theorem}\label{thm:relativeISR}
Let $\Gamma$ be a torsion-free acylindrically hyperbolic group with trivial amenable radical. Then
$\Gamma$ has the relative ISR-property.
\end{theorem}
Alongside the von Neumann algebraic setting, we also establish the same rigidity for reduced group $C^*$-algebras.
\begin{theorem}\label{thm:CstarISR}
Let $\Gamma$ be a torsion-free hyperbolic group. Then $\Gamma$ has relative $C^*$-ISR-property.
\end{theorem}
Motivated by the work of Dudko--Jiang~\cite{dudko2024character}, we also explore the relative ISR-property in the vastly different geometric setting of higher-rank lattices. For irreducible lattices in higher-rank semisimple Lie groups, such as $\Gamma = \mathrm{SL}_d(\mathbb{Z})$ ($d \geq 3$), we establish a corresponding rigidity result for $N$-invariant von Neumann subalgebras.
\begin{theorem}\label{thm:latticeISR}
Let $\Gamma = \mathrm{SL}_d(\mathbb{Z})$ for $d \geq 3$ with $d$ odd. Then $\Gamma$ has the relative ISR-property. 
\end{theorem}
Let us now discuss our proof strategy. A key new ingredient, developed
in Section~\ref{sec:vanishing}, is a \emph{General Vanishing Principle}. This principle is applied
in two distinct geometric settings: for acylindrically hyperbolic
groups and also for higher-rank
lattices.

The proof for torsion-free acylindrically hyperbolic groups then
naturally splits into two cases, depending on whether the Fourier
support of any element of $\cM$ contains a loxodromic element of $N$.
If not, the General Vanishing Principle forces $\cM = \bC$. If so, a
selective averaging principle, a ping-pong construction in a free
subgroup of $N$, and a commutant rigidity argument combine to show
$\cM = L(K)$ for some subgroup $K \leq \Gamma$. The $C^*$-algebraic
case additionally requires $C^*$-irreducible inclusion theory, and is
restricted to torsion-free hyperbolic groups for reasons explained in
Section~\ref{sec:ISR}. The proof of Theorem~\ref{thm:latticeISR} is
structurally different; relying instead on unitary factorization and
the Character Decomposition Property, heavily influenced by \cite{dudko2024character}, is discussed in detail in
Section~\ref{sec:lattice}. 
\subsection*{Acknowledgements} The author thanks Yongle Jiang and Adam Skalski for many helpful discussions and conversations. He also thanks them for taking the time to read through an earlier draft and for their suggestions. He also expresses his gratitude to Yair Glasner for his careful reading of an earlier version of this manuscript and for finding many inaccuracies and suggesting their corrections.

\section{Preliminaries}

\subsection{Acylindrical actions and acylindrically hyperbolic groups}

We recall the basic definitions following Osin~\cite{osin2016acylindrically} and
Dahmani--Guirardel--Osin~\cite{dahmani2017hyperbolically}.

\begin{defn}[Acylindrical action]
An action of a group $\Gamma$ on a metric space $(X, d)$ by isometries is
\emph{acylindrical} if for every $\varepsilon > 0$, there exist constants $R, N > 0$ such
that for all $x, y \in X$ with $d(x, y) \geq R$,
\[
|\{g \in \Gamma : d(x, gx) \leq \varepsilon \text{ and } d(y, gy) \leq \varepsilon\}| \leq N.
\]
A group $\Gamma$ is \emph{acylindrically hyperbolic} if it admits a non-elementary
acylindrical action on a Gromov hyperbolic space.
\end{defn}
A feature we exploit heavily is the rigid structure of maximal elementary
subgroups of individual elements.
\begin{defn}[Elementary closure and primitive elements]
For a loxodromic element $g$ in an acylindrically hyperbolic group $\Gamma$, the
\emph{elementary closure} is
\[
E_\Gamma(g) = \{h \in \Gamma : |\langle g \rangle : \langle g \rangle \cap
h\langle g \rangle h^{-1}| < \infty\}.
\]
In a torsion-free group, $E_\Gamma(g) = C_\Gamma(g)$, which is always infinite cyclic. An
element $g \in \Gamma$ is called \emph{primitive} if $E_\Gamma(g) = \langle g \rangle$,
i.e., $g$ has no proper roots in $\Gamma$.
\end{defn}

Consider a group $\Gamma$ acting on a hyperbolic space $S$. An element $g \in \Gamma$ of infinite order is defined as loxodromic if it possesses exactly two fixed points, $x_g^+$ and $x_g^-$, on the Gromov boundary $\partial S$, with the property that $g^nx \to x_g^+$ for all $x \in \partial S \setminus \{x_g^-\}$. Bestvina and Fujiwara \cite{bestvina2002bounded} introduced the condition of \say{weak proper discontinuity} (WPD). An element $g \in \Gamma$ is called a WPD element if, for any $x \in S$ and $\epsilon > 0$, there exists an integer $M \in \mathbb{N}$ such that only a finite number of elements $h \in \Gamma$ satisfy both $d(x, hx) < \epsilon$ and $d(g^Mx, hg^Mx) < \epsilon$. The connection between these concepts and acylindrically hyperbolic groups was later formalized by Osin \cite[Theorem~1.2]{osin2016acylindrically}. He proved that $\Gamma$ is acylindrically hyperbolic if and only if it contains a loxodromic WPD element $g \in \Gamma$. Furthermore, whenever such a $g$ exists, $E_{\Gamma}(g) = \text{Stab}_{\Gamma}(\{x_g^+, x_g^-\})$ (see, for example, \cite[Lemma~6.5]{dahmani2017hyperbolically}).
\begin{remark}\label{rem:primitive_decomp}
In a torsion-free acylindrically hyperbolic group, every non-trivial element $h \in \Gamma$
can be written as $h = t^k$ for some primitive element $t \in \Gamma$ and integer
$k \geq 1$. Indeed, the maximal amenable subgroup $\Lambda$ containing $h$ satisfies
$\Lambda \cong \mathbb{Z} = \langle t \rangle$ for some primitive $t$, and
$h \in \langle t \rangle$. We refer the reader
to~\cite[Remark~4.2]{amrutam2025invariantC*} for details.
\end{remark}
\begin{remark}
In the context of this paper, we use the term ``primitive'' strictly in the algebraic sense of having no proper roots in the ambient group (i.e., rootlessness). We caution the reader that this differs from the standard usage in combinatorial group theory, where \say{primitive} often denotes an element that can be extended to a free basis of a free group.
\end{remark}

\subsection{Property \texorpdfstring{$P_{naive}$}{} and free subgroups in normal subgroups}

A key geometric input in our arguments is the existence of elements that generate free
subgroups together with any given element. Recall that a group $\Gamma$ has \emph{property
$P_{naive}$} if for any finite subset $F \subset \Gamma \setminus \{e\}$, there exists an
element $s \in \Gamma$ of infinite order such that
$\langle s, t \rangle \cong \langle s \rangle \star \langle t \rangle$ for all $t \in F$.
This property was introduced in~\cite{bekka1994} and established for all acylindrically
hyperbolic groups with trivial amenable radical in~\cite[Theorem~0.2]{AD19}.

In our setting, we are interested in a relative version specific to primitive elements, and the relevant free elements can always be found inside the normal subgroup
$N$. This is the content of the following lemma, which we use in
Section~\ref{sec:cstar}.
\begin{lemma}\label{lem:free_in_normal}
Let $\Gamma$ be a torsion-free acylindrically hyperbolic group with trivial amenable
radical, and let $N \trianglelefteq \Gamma$ be a non-trivial normal subgroup. Then for
every primitive loxodromic element $t \in \Gamma$,
there exists an element $s \in N$ such that
$\langle t, s \rangle \cong \langle t \rangle \star \langle s \rangle$.
\end{lemma}

\begin{proof}
Since $N$ is a non-trivial normal subgroup of a non-elementary acylindrically hyperbolic
group with trivial amenable radical, $N$ is itself non-elementary and acylindrically
hyperbolic by~\cite[Corollary~1.5]{osin2016acylindrically}, and it inherits trivial
amenable radical. Since $N$ contains infinitely many pairwise independent loxodromic WPD elements, arguing similarly as in \cite[Theorem~6.14]{dahmani2017hyperbolically}, we can find $g_1, g_2 \in N$ such that $E_\Gamma(g_i) \cap \langle t \rangle = \{e\}$, and $\mathbb{F}_2\cong\langle g, h\rangle\hookrightarrow_{h} \Gamma$ (in the sense of \cite[Definition~2.9]{osin2016acylindrically}). Now define, for each integer $k \geq 1$, $g^{(k)} := h^k g h^{-k}$.
For each fixed $k$, observe that
\[
\langle g,h \rangle = \langle h \rangle \star \langle g^{(k)} \rangle,
\]
because $h$ and $g^{(k)}$ freely generate $\mathbb{F}_2$. In particular,
$\langle g^{(k)} \rangle$ is a free factor of $\langle g,h \rangle$. By
\cite[Example~2.12(c)]{dahmani2017hyperbolically}, free factors are hyperbolically
embedded, so $\langle g^{(k)} \rangle \hookrightarrow_h \langle g,h \rangle$. Transitivity
of hyperbolic embeddings~\cite[Proposition~4.35]{dahmani2017hyperbolically} then yields
\[
\langle g^{(k)} \rangle \hookrightarrow_h \Gamma.
\]
By almost malnormality of hyperbolically embedded subgroups (\cite[ Proposition 4.33]{dahmani2017hyperbolically}), \cite[Remark 6.2]{dahmani2017hyperbolically}, and \cite[Corollary 6.6]{dahmani2017hyperbolically}, it follows that $E_\Gamma(g^{(k)}) = \langle g^{(k)} \rangle$. Moreover, since $g, h \in N$ and $N  \le \Gamma$, we have $g^{(k)} \in N$ for every
$k$. Because the subgroups $\{\langle g^{(k)} \rangle\}_{k \geq 1}$ are pairwise independent, at most one of them can intersect $\langle t \rangle$ non-trivially. We may therefore fix a $k \geq 1$ such that $\langle g^{(k)} \rangle \cap \langle t \rangle = \{e\}$, and set $r := g^{(k)}$. 

By~\cite[Proposition~2.12]{osin2016acylindrically}, $\Gamma$ is hyperbolic relative to $\{\langle r \rangle\}$. Since $E_\Gamma(t) = \langle t \rangle$ by hypothesis and $\langle r \rangle \cap \langle t \rangle = \{e\}$, \cite[Theorem~4.3, Corollary~1.7]{osin2006elementary} allows us to enlarge the peripheral collection to $\{\langle r \rangle, \langle t \rangle\}$.

We are now in the setting of \cite[Lemma~7]{arzhantseva2006relatively}: $\Gamma$ is
hyperbolic relative to a collection containing the peripheral subgroup $H_\lambda = \langle t \rangle$,
and $r$ is a hyperbolic element with $E_\Gamma(r) = \langle r \rangle$. That lemma produces
$N_0 \in \mathbb{N}$ such that for all $n \geq N_0$,
\[
\langle \langle t \rangle, \, r^n \rangle \cong \langle t \rangle \star \langle r^n \rangle.
\]
Setting $s := r^{N_0}$ (which lies in $N$ because $r \in N$ and $N$ is a subgroup),
the claim follows.
\end{proof}

To handle the case $\cM \cap L(N) \neq \mathbb{C}$, we need a commutant rigidity result
for certain chains of normal subgroups. The following lemma provides the required geometric
input, showing that such chains in an acylindrically hyperbolic group always yield infinite
conjugacy classes and consequently a trivial relative commutant in $L(\Gamma)$. Recall that a subgroup $K\le \Gamma$ is relatively i.c.c.\ if for every $g \in \Gamma \setminus \{e\}$, the conjugacy class
$\{kgk^{-1} \mid k \in K\}$ is infinite. Moreover, we say that a subgroup $\Lambda \leq \Gamma$ is called $s$-normal in $\Gamma$ if for
every $t \in \Gamma$ one has $|\Lambda \cap t^{-1}\Lambda t| = \infty$.

\begin{lemma}\label{lem:plumpness}
Let $\Gamma$ be a torsion-free acylindrically hyperbolic group with trivial amenable radical. Consider the following chain of non-trivial subgroups
$K \trianglelefteq N \trianglelefteq \Gamma$. Then $K$ is relatively i.c.c.\ in $\Gamma$.
In particular, $L(K)' \cap L(\Gamma) = \mathbb{C}$.
\end{lemma}
\begin{proof}
 For any
$g \in \Gamma$, observe that $gKg^{-1} \cap K$ is a normal subgroup of $N$. Since
$\Gamma$ has a trivial amenable radical, it follows that $N$ has a trivial amenable radical. Hence $|gKg^{-1} \cap K| = \infty$,
or trivial. However, if $gKg^{-1} \cap K = \{e\}$, it would imply that
$[K, gKg^{-1}] = \{e\}$. Indeed, since $K \trianglelefteq N$ and $N \trianglelefteq \Gamma$, we have
$gNg^{-1} = N$, and hence $gKg^{-1} \trianglelefteq gNg^{-1} = N$. Thus both
$K$ and $gKg^{-1}$ are normal subgroups of $N$. For any $k \in K$ and
$k' \in gKg^{-1}$, the commutator
\[
[k,k'] = k^{-1}(k')^{-1}kk'
\]
lies in $K$ by normality of $K$ in $N$, and lies in $gKg^{-1}$ by normality of
$gKg^{-1}$ in $N$. Hence $[k,k'] \in K \cap gKg^{-1} = \{e\}$, giving
$[K, gKg^{-1}] = \{e\}$. Choose $k \in K$ a primitive loxodromic element. Then, for any
$t \in \Gamma$ with $tk = kt$, we see that
$t \in \mathrm{Fix}\{x_k^+, x_k^-\} = \langle k \rangle$ (also follows directly from the
proof of~\cite[Corollary~6.9]{osin2016acylindrically}). Consequently, for any
$\tilde{k} \in K$, $g\tilde{k}g^{-1} \in \langle k \rangle \subset K$. Therefore,
$gKg^{-1} \subset K$, a contradiction to $gKg^{-1} \cap K = \{e\}$. This implies
$|gKg^{-1} \cap K| = \infty$, thereby showing that $K$ is $s$-normal. Using
\cite[Lemma~7.1]{osin2016acylindrically}, we see that the acylindrical non-elementary
action $\Gamma \curvearrowright S$ restricted to $K$ is non-elementary. Therefore, using
\cite[Theorem~1.1]{osin2016acylindrically}, we see that $K$ contains infinitely many
independent loxodromic elements. To show that $K$ is relatively i.c.c.\ in $\Gamma$, by the orbit-stabilizer theorem,
we must show that the centralizer $C_K(g)$ has infinite index in $K$.

Assume for the sake of contradiction that there exists a non-trivial element
$g \in \Gamma \setminus \{e\}$ such that $[K : C_K(g)] < \infty$. Let
$C = \mathrm{Core}_K(C_K(g))$ be the normal core of the centralizer inside $K$. Moreover,
$[K : C] < \infty$. By construction, $C$ is a non-trivial normal subgroup of $K$. Since
$K$ is acylindrically hyperbolic (see~\cite[Corollary~1.5]{osin2016acylindrically}), $C$
is $s$-normal in $K$. Therefore, by~\cite[Theorem~1.1]{osin2016acylindrically}, $C$
contains infinitely many independent loxodromic elements. Let $s, t \in C$ be two such
independent loxodromic elements such that
$\mathrm{Fix}\{x_s^+, x_s^-\} \cap \mathrm{Fix}\{x_t^+, x_t^-\} = \emptyset$. Because
$C \subseteq C_K(g)$, both $s$ and $t$ commute with $g$. In an acylindrically hyperbolic
group, the centralizer of any loxodromic element is contained in its maximal elementary
subgroup, giving us that
\[
g \in C_\Gamma(s) \cap C_\Gamma(t) \subseteq E(s) \cap E(t)
= \langle s \rangle \cap \langle t \rangle = \{e\}.
\]
This is a contradiction. The final statement $L(K)' \cap L(\Gamma) = \mathbb{C}$ follows
from the relative i.c.c.\ property by~\cite[Theorem~3.7]{jiang2021maximal}.
\end{proof}
\begin{remark} One can even argue similarly as in \cite[Lemma~3.5]{amrutam2021intermediate} to show that $K$ is plump in $\Gamma$.
\end{remark}

\subsection{Group \texorpdfstring{$C^*$}{}-algebras and von Neumann algebras}

For a countable discrete group $\Gamma$, the reduced group $C^*$-algebra $C_r^*(\Gamma)$
is generated inside $\mathbb{B}(\ell^2\Gamma)$ by the left regular representation
$\lambda: \Gamma \to \mathbb{B}(\ell^2\Gamma)$. It carries the canonical tracial state
$\tau_0$ determined by $\tau_0(\lambda(g)) = \delta_{g,e}$. The group von Neumann algebra
$L(\Gamma)$ is the weak operator closure of $C_r^*(\Gamma)$ inside $\mathbb{B}(\ell^2\Gamma)$.
Throughout this paper, we use $\cA$ to denote a unital $C^*$-subalgebra of $C_r^*(\Gamma)$ and
$\cM$ to denote a von Neumann subalgebra of $L(\Gamma)$.

Recall that for any subgroup $\Lambda \le \Gamma$, there is a canonical, faithful, normal conditional expectation $\mathbb{E}_\Lambda : L(\Gamma) \to L(\Lambda)$ determined by its action on the generating unitaries:
\[
\mathbb{E}_\Lambda\left(\lambda(s)\right)=\begin{cases}
\lambda(s) & \mbox{if } s \in \Lambda \\
0 & \mbox{otherwise.}
\end{cases}
\]
For a single element $g \in \Gamma$, we write $\mathbb{E}_g$ for the expectation onto $L(\langle g \rangle)$. 
Similarly, for any subgroup
$H \leq \Gamma$, there is a canonical conditional expectation
$\E_H : C_r^*(\Gamma) \to C_r^*(H)$. We
write $c_\gamma(a) := \tau_0(a\lambda(\gamma)^*)$ for the Fourier coefficient of an
element $a \in C_r^*(\Gamma)$ (or $a \in L(\Gamma)$) at $\gamma \in \Gamma$.

\subsection{Furstenberg boundary and invariant amenable subalgebras}
In this subsection, we establish that invariant amenable subalgebras within the appropriate $C^*$-simple geometric settings must be trivial. This automatically guarantees that any $N$-invariant subalgebra possesses a trivial center. To achieve this, we make use of the Furstenberg boundary.

The Furstenberg boundary, denoted by $\partial_F\Gamma$, is the universal boundary associated with the group $\Gamma$, in the sense that any other $\Gamma$-boundary $Y$ can be obtained as a $\Gamma$-equivariant continuous image of it. The existence of this space is typically established via a standard product construction over all possible boundaries (see, e.g., \cite[p.~199]{Furstenberg1973}), and it is uniquely determined up to $\Gamma$-equivariant homeomorphism. 

The dynamics of the $\Gamma$-action on $\partial_F\Gamma$ encode deep structural information about the group itself. For instance, Kalantar and Kennedy \cite{KalantarKennedy2017} utilized this boundary action to provide a purely dynamical characterization of $C^*$-simplicity. Furthermore, a classical result states that $\Gamma$ is an amenable group if and only if its Furstenberg boundary consists of a single point (see, e.g., \cite[Theorem~3.1, Chapter~3]{Glasner1976}). Generalizing this, Furman \cite[Proposition~7]{Furman2003} proved that the amenable radical, $\operatorname{Rad}(\Gamma)$, coincides precisely with the kernel of the boundary action $\Gamma \curvearrowright \partial_F\Gamma$ (also see \cite[Proposition~2.8]{breuillard2017c}). This specific characterization serves as a primary tool for our subsequent arguments.

Finally, we note that the induced affine action $\Gamma \curvearrowright \operatorname{Prob}(\partial_F\Gamma)$ is irreducible. This means there are no non-trivial, weak*-closed, $\Gamma$-invariant convex subsets within $\operatorname{Prob}(\partial_F\Gamma)$.

We now establish the rigidity of amenable subalgebras invariant under normal subgroups. The argument is vis-à-vis~\cite [Proposition~3.3]{amrutam2025amenable}, and we modify it wherever needed. 
\begin{prop}\label{prop:amenable_subalgebra}
Let $\Gamma$ be a countable discrete group and let $N \trianglelefteq \Gamma$ be a $C^*$-simple normal subgroup such that $C_\Gamma(N) = \{e\}$. Then every $N$-invariant amenable von Neumann subalgebra $\cM \le L(\Gamma)$ is trivial, i.e., $\cM = \bC$.
\end{prop}
\begin{proof}
Let $\cM \le L(\Gamma)$ be an $N$-invariant amenable von Neumann subalgebra. Because $\cM$ is amenable, the set of $\cM$-hypertraces on $\mathbb{B}(\ell^2(\Gamma))$ whose restriction to $L(\Gamma)$ is the canonical trace $\tau_0$ is non-empty (see~\cite[Proposition~2.4]{amrutam2025amenable}). We denote this collection by $\text{Hype}_{\tau_0}(\cM)$. Using \cite[Lemma~5.2]{breuillard2017c}, the action of $N$ on its universal Furstenberg boundary $\partial_F N$ extends to a $\Gamma$-boundary action on $\partial_F N$ such that the action $\Gamma\curvearrowright\partial_FN$ is free~(\cite[Lemma~5.3]{breuillard2017c}). We can view $C(\partial_F N)$ as a subalgebra of multiplication operators inside $\mathbb{B}(\ell^2(\Gamma))$. Since $\cM$ is $N$-invariant, the restriction $\text{Hype}_{\tau_0}(\cM)|_{C(\partial_F N)}$ forms an $N$-invariant, weak*-closed, convex subset of $\text{Prob}(\partial_F N)$. Since the action $N \curvearrowright \text{Prob}(\partial_F N)$ is irreducible, we see that $\text{Hype}_{\tau_0}(\cM)|_{C(\partial_F N)} = \text{Prob}(\partial_F N)$. In particular, for every $x \in \partial_F N$, the Dirac measure $\delta_x$ is the restriction of some hypertrace $\varphi \in \text{Hype}_{\tau_0}(\cM)$.

To show that $\cM = \bC$, we take an arbitrary element $a \in \cM$ and show that for any $s \in \Gamma \setminus \{e\}$, $\tau_0(a \lambda(s)^*) = 0$. 

Fix $s \in \Gamma \setminus \{e\}$. Since $\Gamma\curvearrowright \partial_FN$ is free, $sx\ne x$ for all $x\in\partial_FN$. Fix $x_0\in\partial_FN$.
Since $s x_0 \neq x_0$, we choose a function $f \in C(\partial_F N)$ such that $f(x_0) = 1$, $f(s x_0) = 0$, and $0 \le f \le 1$. By the irreducibility established above, we can find an $\cM$-hypertrace $\varphi \in \text{Hype}_{\tau_0}(\cM)$ such that $\varphi|_{C(\partial_F N)} = \delta_{x_0}$.

Since $C(\partial_F N)$ falls in the multiplicative domain of $\varphi$, \(\varphi\) is an \(\cM\)-hypertrace; that is, \(\varphi(aT)=\varphi(Ta)\) for every \(T\in\mathbb{B}(\ell^2(\Gamma))\) (see~\cite[Proposition~2.4]{amrutam2025amenable}), we have
\begin{align*}
    \tau_0(a \lambda(s)^*) &= \varphi(a \lambda(s)^*) \\
    &= \varphi(a \lambda(s)^* f) \\
    &= \varphi(a (s^{-1} \cdot f) \lambda(s)^*) \\&=\varphi( (s^{-1} \cdot f) \lambda(s)^*a)\\&=f(sx_0)\varphi(\lambda(s)^*a)\\&=0.
\end{align*}
Thus, $\tau_0(a \lambda(s)^*) = 0$ for all $s \in \Gamma \setminus \{e\}$. This forces the Fourier support of $a$ to be concentrated solely at the identity, yielding $a \in \bC$. The proof is complete.
\end{proof}

\begin{cor}\label{cor:trivial_center}
Let $\Gamma$ be a countable discrete group and $N \trianglelefteq \Gamma$ be a $C^*$-simple normal subgroup such that $C_\Gamma(N) = \{e\}$. If $\cM \le L(\Gamma)$ is an $N$-invariant von Neumann subalgebra, then its center is trivial, i.e., $\mathcal{Z}(\cM) = \bC$.
\end{cor}
\begin{proof}
The center $\mathcal{Z}(\cM)$ is an abelian, and hence amenable, von Neumann subalgebra of $L(\Gamma)$. Since $\cM$ is normalized by $N$, its center $\mathcal{Z}(\cM)$ is clearly also $N$-invariant. Applying Proposition~\ref{prop:amenable_subalgebra} directly yields $\mathcal{Z}(\cM) = \bC$.
\end{proof}

\subsection{Higher-Rank Lattices and Characters}
In Section~\ref{sec:lattice}, our focus shifts to irreducible lattices in higher-rank semisimple Lie groups, for example, $\Gamma = \mathrm{SL}_d(\mathbb{Z})$ with $d \geq 3$. We will utilize the following fundamental theorem regarding their normal subgroup structure.
\begin{theorem}[Margulis Normal Subgroup Theorem]\label{thm:margulis}
Let $\Gamma$ be an irreducible lattice in a higher-rank semisimple Lie group
with trivial center and no compact factors (such as $\mathrm{SL}_d(\mathbb{Z})$
for $d \geq 3$ odd). Then every normal subgroup of $\Gamma$ is either finite or
of finite index in $\Gamma$. In particular, when $\Gamma =
\mathrm{SL}_d(\mathbb{Z})$ with $d \geq 3$ odd, every \emph{non-trivial} normal
subgroup has finite index, since $\Gamma$ admits no non-trivial finite normal
subgroups.
\end{theorem}
Furthermore, we will rely on characters and their connections to operator algebras in the higher-rank setup. A function $\phi: \Gamma \to \mathbb{C}$ is a \emph{character} if it is normalized ($\phi(e) = 1$), positive-definite, and constant on conjugacy classes of $\Gamma$. 
Following~\cite[Definition~3.1]{dudko2024character}, a countable group $G$ is
said to have the \emph{non-factorizable character decomposition property} (CDP)
if for any two characters $\phi, \psi \colon G \to \mathbb{C}$ -- that is,
normalized ($\phi(e)=\psi(e)=1$), positive-definite functions constant on
conjugacy classes of $G$ -- satisfying $\phi(s)\psi(s) = 0$ for all
$e \neq s \in G$, one has either $\phi \equiv \delta_e$ or $\psi \equiv
\delta_e$. Building on the operator-algebraic superrigidity of
Bekka~\cite{bekka2007operator}, the group $\mathrm{SL}_d(\mathbb{Z})$
($d \geq 3$ odd) has the
CDP~\cite[Proposition~3.17]{dudko2024character}. Moreover,
by~\cite[Proposition~3.2(1)]{dudko2024character}, every non-trivial normal
subgroup of a group with CDP again has the CDP. In particular, every
finite-index normal subgroup $N \trianglelefteq \mathrm{SL}_d(\mathbb{Z})$
inherits the CDP from $\Gamma$.
\section{A General Vanishing Principle}\label{sec:vanishing}

The arguments in both the acylindrically hyperbolic and higher-rank
lattice settings share a common analytic core: a mechanism for forcing
the conditional expectation $\mathbb{E}_{\cM}(\lambda(g))$ to vanish
for every non-identity group element, thereby concluding $\cM = \bC$.

Throughout this section, $\Gamma$ denotes a countable discrete group,
$N \trianglelefteq \Gamma$ a non-trivial normal subgroup, and
$\cM \leq L(\Gamma)$ an $N$-invariant von Neumann subalgebra.  With such $\cM$, we can associate a positive definite function $\phi_\cM \colon \Gamma \to [0,\infty)$, defined by
\[
  \phi_\cM(g)
  := \tau_0\!\left(\mathbb{E}_\cM(\lambda(g))\,\lambda(g)^*\right).
\]
We record the basic, well-known properties of this function for our later use. 
\begin{lemma}\label{lem:phi_properties}
Let $\cM \leq L(\Gamma)$ be an $N$-invariant von Neumann subalgebra.
\begin{enumerate}[label={\rm(\arabic*)}]
\item $\phi_\cM(g) = \|\mathbb{E}_\cM(\lambda(g))\|_2^2$ for all
  $g \in \Gamma$. In particular, $\phi_\cM(g) \geq 0$,
  $\phi_\cM(e) = 1$, and $\phi_\cM(g) = 0$ if and only if
  $\mathbb{E}_\cM(\lambda(g)) = 0$.
\item $\phi_\cM(g^{-1}) = \phi_\cM(g)$ for all $g \in \Gamma$.
\item $\phi_\cM$ is $N$-invariant: $\phi_\cM(ngn^{-1}) = \phi_\cM(g)$
  for all $n \in N$ and $g \in \Gamma$.
\item $\phi_\cM$ is a positive-definite function on $\Gamma$.
\end{enumerate}
\end{lemma}
\begin{proof}
See~\cite[Proposition~3.2]{jiang2024example}.
\end{proof}
Since $\phi_\cM$ is positive-definite with $\phi_\cM(e) = 1$, we apply the GNS construction to obtain a unitary representation
$(\pi,\mathcal{H},\xi)$ of $\Gamma$ with a unit cyclic vector $\xi$
satisfying
\begin{equation}\label{eq:GNS}
  \bigl\langle \pi(g)\,\xi,\,\pi(h)\,\xi \bigr\rangle_{\mathcal{H}}
  = \phi_\cM(h^{-1}g)
  \qquad\text{for all }g,h\in\Gamma.
\end{equation}
In particular, $\langle \pi(g)\xi, \xi\rangle = \phi_\cM(g)$ and
$\|\pi(g)\xi\| = 1$ for all $g \in \Gamma$.

Recall that the Fourier support of $a \in L(\Gamma)$ is
$\operatorname{Supp}(a) = \{g \in \Gamma : c_g(a) \neq 0\}$
where $c_g(a) := \tau_0(a\lambda(g)^*)$.

We now introduce the two main tools of this section. The first is a
pointwise criterion, the \emph{Support Principle}, which translates a
Fourier-support assumption directly into the vanishing of
$\mathbb{E}_\cM(\lambda(g))$ for loxodromic elements. The second is an
abstract Bessel-type argument, formalized via the notion of thickness,
which propagates this vanishing to all of $\Gamma \setminus \{e\}$.
\begin{lemma}[Support Principle]\label{lem:support_principle}
Let $\Gamma$ be a torsion-free acylindrically hyperbolic group and let
$\cM \leq L(\Gamma)$ be a von Neumann subalgebra. Let $g\in\Gamma$ be a loxodromic element. Suppose that
$\operatorname{Supp}(a)$ does not contain $g$ for 
any $a \in \cM$. Then $\mathbb{E}_\cM(\lambda(g)) = 0$.
\end{lemma}
\begin{proof}
Let $g \in \Gamma$ be loxodromic. The element
$\mathbb{E}_\cM(\lambda(g))$ lies in $\cM$, so by hypothesis
$g \notin \operatorname{Supp}(\mathbb{E}_\cM(\lambda(g)))$. Therefore,
\[\|\mathbb{E}_\cM(\lambda(g))\|_2^2= \tau_0\!\left(\mathbb{E}_\cM(\lambda(g))\,\lambda(g)^*\right) = 0.
\]
\end{proof}
\begin{remark}
Assume further that $\operatorname{Supp}(a)$ does not contain $g^k$ for 
any $a \in \cM$ and $k\in\mathbb{Z}\setminus\{0\}$. Then the same conclusion extends to all non-zero powers of $g$. Indeed, we observe that since
$g$ is loxodromic, it acts on the Gromov-hyperbolic space $S$ with
exactly two fixed points $\{g^+,g^-\}\subset\partial S$. Every non-zero
power $g^k$ ($k\neq 0$) shares the same fixed-point pair and is
therefore also loxodromic. Thus, the hypothesis applies equally to each
$g^k$, and the argument above gives $\mathbb{E}_\cM(\lambda(g^k)) = 0$
for all $k \neq 0$.    
\end{remark}
\begin{remark}\label{rem:hyperbolic_trivial}
In a torsion-free hyperbolic group, every non-trivial
element is loxodromic. Hence, in that setting, the hypothesis of the
Support Principle immediately yields $\mathbb{E}_\cM(\lambda(g)) = 0$
for every $g \in \Gamma\setminus\{e\}$, and therefore $\cM = \bC$,
without any further argument. In the broader acylindrically hyperbolic
setting, where infinite-order elliptic elements may exist, the Support
Principle handles loxodromic elements only, and the GNS Vanishing Lemma
below is required to complete the argument.
\end{remark}
The Support Principle handles loxodromic elements individually. To
upgrade this to a statement about all of $\Gamma \setminus \{e\}$, we
need a mechanism for propagating vanishing across the group. This is
achieved through the following notion, which captures the idea that
conjugates of any group element by elements of $N$ are sufficiently
spread out inside a prescribed set.
\begin{definition}\label{def:thick}
Let $N \trianglelefteq \Gamma$ be a non-trivial normal subgroup and
$\cM \leq L(\Gamma)$ an $N$-invariant von Neumann subalgebra. A
non-empty subset $K \subseteq \Gamma \setminus \{e\}$ is called
\emph{$(\cM, N)$-thick} if for every $g \in \Gamma \setminus \{e\}$,
there exists an infinite sequence $\{n_i\}_{i=1}^\infty \subset N$
such that
\[
w_{i,j} := n_j^{-1}\,g\,n_j \cdot n_i^{-1}\,g^{-1}\,n_i \in K
\qquad \text{for all } i \neq j.
\]
\end{definition}
\begin{lemma}[GNS Vanishing Lemma]\label{lem:GNS_vanishing}
Let $\Gamma$ be a countable discrete group, $N \trianglelefteq \Gamma$
a non-trivial normal subgroup, and $\cM \leq L(\Gamma)$ an $N$-invariant
von Neumann subalgebra. Suppose there exists an $(\cM, N)$-thick subset
$K \subseteq \Gamma \setminus \{e\}$ such that
$\mathbb{E}_\cM(\lambda(s)) = 0$ for every $s \in K$. Then $\phi_\cM
\equiv 0$ on $\Gamma \setminus \{e\}$. Consequently, $\cM = \bC$.
\end{lemma}

\begin{proof}
Fix $g \in \Gamma \setminus \{e\}$. Since $K$ is $(\cM, N)$-thick,
there exists an infinite sequence $\{n_i\}_{i \geq 1} \subset N$ such
that $w_{i,j} \in K$ for all $i \neq j$. Let $(\pi, \mathcal{H}, \xi)$
be the GNS triple of $\phi_\cM$ from~\eqref{eq:GNS}. Define
\[
  v_i := \pi\!\left(n_i^{-1}\,g^{-1}\,n_i\right)\xi
  \in \mathcal{H}, \qquad i \geq 1.
\]
Since $\pi$ is unitary and $\|\xi\|=1$, we have $\|v_i\|=1$ for all
$i$. For $i \neq j$, the GNS formula~\eqref{eq:GNS} gives
\begin{align*}
  \langle v_i, v_j \rangle
  &= \phi_\cM\!\left(
     \bigl(n_j^{-1}g^{-1}n_j\bigr)^{-1}
     \bigl(n_i^{-1}g^{-1}n_i\bigr)\right)
  = \phi_\cM(w_{i,j}).
\end{align*}
Since $w_{i,j} \in K$ and $\mathbb{E}_\cM(\lambda(w_{i,j})) = 0$,
Lemma~\ref{lem:phi_properties}(1) gives $\phi_\cM(w_{i,j}) = 0$,
so $\langle v_i, v_j \rangle = 0$. By
Lemma~\ref{lem:phi_properties}(2) and~(3),
\[
  |\langle v_i, \xi \rangle|^2
  = \phi_\cM(n_i^{-1}g^{-1}n_i)^2
  = \phi_\cM(g)^2
  \qquad \text{for all } i \geq 1.
\]
Bessel's inequality applied to the orthonormal sequence
$\{v_i\}_{i=1}^k$ gives
\[
  k\,\phi_\cM(g)^2
  = \sum_{i=1}^{k}|\langle v_i, \xi\rangle|^2
  \leq \|\xi\|^2 = 1
  \qquad\text{for all }k \geq 1.
\]
Letting $k\to\infty$ forces $\phi_\cM(g) = 0$. Since $g \in \Gamma
\setminus \{e\}$ was arbitrary, Lemma~\ref{lem:phi_properties}(1)
gives $\mathbb{E}_\cM(\lambda(g)) = 0$ for all $g \neq e$, and
consequently $\cM = \bC$.
\end{proof}
It remains to verify that the set of loxodromic elements of $N$ is
indeed $(\cM, N)$-thick in the acylindrically hyperbolic setting. This
is the content of the following proposition, which is the key geometric
input of this section. 
\begin{prop}\label{prop:w_ij_loxodromic}
Let $\Gamma$ be a torsion-free acylindrically hyperbolic group with
trivial amenable radical, and let $N \trianglelefteq \Gamma$ be a
non-trivial normal subgroup. For every $g \in \Gamma \setminus \{e\}$,
there exists an infinite sequence $\{n_i\}_{i \geq 1} \subset N$ of
loxodromic elements such that
\[
  w_{i,j} = n_j^{-1}\,g\,n_j \cdot n_i^{-1}\,g^{-1}\,n_i
\]
is loxodromic for all $i \neq j$.
\end{prop}

\begin{proof}
Since $N$ is a non-trivial normal subgroup of a non-elementary
acylindrically hyperbolic group with trivial amenable radical, $N$ is
itself non-elementary and acylindrically hyperbolic
by~\cite[Corollary~1.5]{osin2016acylindrically}.
Since $N$ contains infinitely many pairwise independent loxodromic
elements, choose one, call it $h \in N$. Since $\Gamma$ is torsion-free, the maximal elementary subgroup satisfies
$E_\Gamma(h)=\mathrm{Stab}_\Gamma(\{x_{h^+},x_{h^-}\})$. Write $h=h_0^n$ for some primitive element $h_0\in\Gamma$. 
If $g$ is loxodromic, then $g \in E_\Gamma(h) = \langle h_0 \rangle$
only if $g$ and $h$ generate the same maximal cyclic subgroup, i.e.,
$\{x_g^+, x_g^-\} = \{x_h^+, x_h^-\}$. Since $N$ contains infinitely many
pairwise independent loxodromic elements, at most one of them can share
the fixed-point pair $\{x_g^+, x_g^-\}$ with $g$. We therefore choose $h
\in N$ loxodromic with $\{x_h^+, x_h^-\} \cap \{x_g^+, x_g^-\}=\emptyset$.

Assume now that $g$ is elliptic.
We claim $g \notin E_\Gamma(h)$. Indeed, if
$g \in E_\Gamma(h)=\langle h_0\rangle$, then $g$ is either the identity
or a non-zero power of $h_0$, which
contradicts the assumption that $g$ is elliptic. 

Let
$\alpha := h^{-1}$ and $\beta := ghg^{-1}$. Both are loxodromic, with
fixed-point pairs $\{x_h^+,x_h^-\}$ and $\{gx_{h}^+,gx_h^-\}$ respectively, and
these pairs are disjoint. Hence $\alpha$ and $\beta$ are independent
loxodromic elements.

Using~\cite[Lemma~1.2]{AD19}, since $\alpha$ and
$\beta$ are independent loxodromic elements, there exists $M \in \bN$
such that for all integers $k \geq M$, the product $\alpha^k\beta^k$ is
loxodromic. Observe that
\[
  \alpha^k\beta^k
  = h^{-k}\cdot(ghg^{-1})^k
  = h^{-k}\cdot gh^kg^{-1}
  = h^{-k}gh^kg^{-1}.
\]
Hence $h^{-k}gh^kg^{-1}$ is loxodromic for every $k \geq M$.
Define $n_i := h^{iM} \in N$ for $i \geq 1$. Since $h \in N$ and $N$
is a subgroup, $n_i \in N$ for all $i \geq 1$. Moreover, each $n_i$ is
loxodromic, as it is a non-zero power of the loxodromic element $h$. Fix $i \neq j$. Without loss of generality, assume $j > i$. We see that
\begin{align*}
  w_{i,j}
  = n_j^{-1}\,g\,n_j \cdot n_i^{-1}\,g^{-1}\,n_i 
  = h^{-jM}\,g\,h^{jM} \cdot h^{-iM}\,g^{-1}\,h^{iM}
  &= h^{-jM}\,g\,h^{(j-i)M}\,g^{-1}\,h^{iM}.
\end{align*}
Setting $k := (j-i)M \geq M$ (since $j > i$), we rewrite this as
\[
  w_{i,j}
  = h^{-iM}\cdot\bigl(h^{-k}\,g\,h^k\,g^{-1}\bigr)\cdot h^{iM}.
\]
The element in parentheses is $h^{-k}gh^kg^{-1}$, which is loxodromic since $k \geq M$. Since $w_{i,j}$ is a conjugate of a
loxodromic element by $h^{iM}$, and loxodromicity is invariant under
conjugation, $w_{i,j}$ is loxodromic. If instead $i > j$, then
$w_{i,j} = w_{j,i}^{-1}$, which is again loxodromic since
loxodromicity is closed under taking inverses. In either case, since $N$ is normal, $gh^kg^{-1}\in N$ and hence,
$w_{i,j}$ is a loxodromic element of $N$.
\end{proof}
With the thickness of the loxodromic set established, the two main
corollaries of this section follow by combining the Support Principle
with the GNS Vanishing Lemma. The first handles the von Neumann
algebraic setting directly.
\begin{cor}\label{cor:trivial_lox_support}
Let $\Gamma$ be a torsion-free acylindrically hyperbolic group with
trivial amenable radical, $N \trianglelefteq \Gamma$ a non-trivial
normal subgroup, and $\cM \leq L(\Gamma)$ an $N$-invariant von Neumann
subalgebra. If $\operatorname{Supp}(a)$ contains no loxodromic element
of $N$ for any $a \in \cM$, then $\cM = \bC$.
\end{cor}
\begin{proof}
It follows from Lemma~\ref{lem:support_principle} that
$\mathbb{E}_\cM(\lambda(n)) = 0$ for every loxodromic $n \in N$.
Taking $K$ to be the set of all loxodromic elements of $N$, we have
$\mathbb{E}_\cM(\lambda(s)) = 0$ for every $s \in K$. By
Proposition~\ref{prop:w_ij_loxodromic}, $K$ is $(\cM, N)$-thick. Using Lemma~\ref{lem:GNS_vanishing}, we get that
$\cM = \bC$.
\end{proof}
The $C^*$-algebraic analogue follows by passing to the weak closure,
where the vanishing condition transfers by weak continuity of the canonical trace.
\begin{cor}\label{cor:cstar_trivial_lox_support}
Let $\Gamma$ be a torsion-free acylindrically hyperbolic group with
trivial amenable radical, $N \trianglelefteq \Gamma$ a non-trivial
normal subgroup, and $\cA \leq C_r^*(\Gamma)$ a unital $N$-invariant
$C^*$-subalgebra. If $\tau_0(a\lambda(n)^*) = 0$ for every $a \in \cA$
and every loxodromic $n \in N$, then $\cA = \bC$.
\end{cor}
\begin{proof}
Let $\cM := \cA'' \leq L(\Gamma)$ denote the weak operator closure of
$\cA$ inside $L(\Gamma)$. Clearly, $\cM$ is $N$-invariant.
We claim that $\tau_0(x\lambda(n)^*) = 0$ for every $x \in \cM$ and
every loxodromic $n \in N$. Fix a loxodromic $n \in N$ and let
$x \in \cM$. By definition of the weak operator closure, there exists a
net $\{a_\alpha\} \subset \cA$ such that $a_\alpha \to x$ in the weak
operator topology. The functional
$y \mapsto \tau_0(y\lambda(n)^*) = \langle y\delta_e,
\lambda(n)\delta_e\rangle$ is weak-operator continuous. Hence
\[
    \tau_0(x\lambda(n)^*)
    = \lim_\alpha \tau_0(a_\alpha\lambda(n)^*)
    = 0,
\]
Hence, no loxodromic element of $N$ belongs to
$\operatorname{Supp}(x)$ for any $x \in \cM$. Since $\cM$ is an
$N$-invariant von Neumann subalgebra of $L(\Gamma)$,
Corollary~\ref{cor:trivial_lox_support} applies and gives $\cM = \bC$.
Since $\cA \subseteq \cM = \bC$ and $\cA$ is unital, we obtain
$\cA = \bC$.
\end{proof}
\section{Invariant subalgebras: the averaging argument}
\label{sec:cstar}

Throughout this section, let $\Gamma$ be a torsion-free acylindrically hyperbolic group unless otherwise mentioned, and $N \trianglelefteq \Gamma$ a non-trivial non-elementary normal subgroup. Let $\cA \leq C_r^*(\Gamma)$ be a unital $N$-invariant $C^*$-subalgebra, i.e.,
$\lambda(n)\cA\lambda(n)^* = \cA$ for all $n \in N$.

The first step is a selective averaging principle that allows us to project $\cA$ onto
cyclic subgroup algebras while staying inside $\cA$. The key point is that although $\cA$
is only $N$-invariant rather than $\Gamma$-invariant, we can still average by powers of
any primitive element $t$, provided a power of $t$ lands in $N$. This tool has been used before in \cite{amrutam2025invariantC*} and was first introduced in \cite{amrutam2024relative}.
\begin{prop}[Selective averaging]\label{prop:selective_averaging_N}
Let $\cA \leq C_r^*(\Gamma)$ be an $N$-invariant $C^*$-subalgebra. Let $t \in \Gamma$ be
a loxodromic primitive element such that $t^k \in N$ for some $k \geq 1$. Then
$\E_t(\cA) \subseteq \cA$. Similarly, for an $N$-invariant von Neumann algebra $\cM\le L(\Gamma)$, we have $\mathbb{E}_t(\cM)\subset\cM$.
\end{prop}

\begin{proof}
Let $a \in \cA$ and $\epsilon > 0$. Since $t^k \in N$ and $N$ is a subgroup, the
subsequence $\{t^{kj} : j \geq 1\} \subseteq N$. By $N$-invariance of $\cA$:
\[
\lambda(t^{kj})\,a\,\lambda(t^{kj})^* \in \cA \qquad \forall\, j \geq 1.
\]
Since $t$ is primitive in $\Gamma$, using~\cite[Lemma~4.1]{amrutam2025invariantC*}, we see that
$\mathrm{Stab}_{\Gamma}(\{t^+, t^-\}) = \langle t \rangle$, so in particular
$h\{t^+, t^-\} \cap \{t^+, t^-\} = \emptyset$ for all $h \notin \langle t \rangle$.
Applying Proposition~3.2 of~\cite{amrutam2025invariantC*} to the element $t$ and the
subsequence $\{kj : j \geq 1\}$, we obtain elements
$t^{kj_1}, \ldots, t^{kj_m} \in \{t^{kj} : j \geq 1\}$ such that
\[
\left\|\frac{1}{m}\sum_{i=1}^{m} \lambda(t^{kj_i})\,a\,\lambda(t^{kj_i})^* - \E_t(a)\right\|
< \epsilon.
\]
Since each $\lambda(t^{kj_i})a\lambda(t^{kj_i})^* \in \cA$ and $\cA$ is norm-closed,
$\E_t(a) \in \cA$.
\end{proof}

With selective averaging at hand, we can now run a ping-pong argument inside a free
subgroup to extract individual group unitaries from $\cA$. The free element required for
the ping-pong is supplied by Lemma~\ref{lem:free_in_normal}, and its normality in $\Gamma$
ensures all conjugations stay within the $N$-invariant subalgebra. 

\begin{remark}
The proof technique of the following theorem is essentially similar to the approach developed by Amrutam and Jiang \cite{amrutam2023invariant}. The major difference in our setting is that the free element required for the ping-pong argument is specifically chosen from the non-trivial normal subgroup $N$, which ensures that all relevant conjugations respect the $N$-invariance of the subalgebra.
\end{remark}

\begin{theorem}\label{thm:N_invariant_position}
Let $\cA \leq C_r^*(\Gamma)$ be an $N$-invariant unital $C^*$-subalgebra. Let
$g \in N$ be a loxodromic element and $a \in \cA$ be such that $\tau_0(a\lambda(g)^*) \neq 0$.
Then $\lambda(g) \in \cA$. A similar conclusion holds for an $N$-invariant von Neumann subalgebra $\cM\le L(\Gamma)$.
\end{theorem}
\begin{proof}
We provide the proof for $\mathcal{A}$, and the proof for $\mathcal{M}$ is analogous. 
We may assume without loss of generality that $\tau_0(a) = 0$. Write
$c_\gamma := \tau_0(a\lambda(\gamma)^*)$ for $\gamma \in \Gamma$, so $c_g \neq 0$ by
assumption. We often write $g$ instead of $\lambda(g)$ for ease of notation. Write
$g = t^k$ for some primitive element $t \in \Gamma$ and integer $k \geq 1$ (this is
possible by Remark~4.2 of~\cite{amrutam2025invariantC*} applied in the torsion-free
setting). Since $t^k = g \in N$, Proposition~\ref{prop:selective_averaging_N} gives
\begin{equation}\label{eq:Et_in_A}
\E_t(a) = \sum_{m \in \mathbb{Z}} c_{t^m}\lambda(t^m) \in \cA.
\end{equation}
Note that $c_{t^k} = c_g \neq 0$. Using Lemma~\ref{lem:free_in_normal}, there exists an element $s \in N$ such that
\[
\langle t, s \rangle \cong \langle t \rangle \star \langle s \rangle \cong F_2.
\]
Since $s \in N$, $N$-invariance gives $\lambda(s)a\lambda(s)^* \in \cA$. The element
$sts^{-1} \in \Gamma$ is primitive (conjugation preserves primitivity), and
\[
(sts^{-1})^k = st^ks^{-1} = sgs^{-1} \in N,
\]
since $g \in N$, $s \in N$, and $N \trianglelefteq \Gamma$. Using
Proposition~\ref{prop:selective_averaging_N} applied to $sts^{-1}$ and
$\lambda(s)a\lambda(s)^*$, we get that
\begin{equation}\label{eq:Ests_in_A}
\E_{sts^{-1}}\left(\lambda(s)a\lambda(s)^*\right)
= \sum_{m \in \mathbb{Z}} c_{t^m}\lambda(st^ms^{-1}) \in \cA.
\end{equation}
Indeed, $\E_{sts^{-1}}(\lambda(s)\lambda(t^m)\lambda(s)^*) = \lambda(st^ms^{-1})$, and
since $E_{sts^{-1}}(\lambda(s)\lambda(\gamma)\lambda(s)^*) = 0$ for
$\gamma \notin \langle t \rangle$ (as $s\gamma s^{-1} \notin \langle sts^{-1} \rangle$
for $\gamma \notin \langle t \rangle$ in $F_2$), the formula follows. Multiplying
\eqref{eq:Et_in_A} and \eqref{eq:Ests_in_A} inside $\cA$, we get that
\begin{equation}\label{eq:product_y}
y := \E_t(a)\cdot \E_{sts^{-1}}\left(\lambda(s)a\lambda(s)^*\right)
= \sum_{k_1, k_2 \in \mathbb{Z}} c_{t^{k_1}} c_{t^{k_2}}\,\lambda(t^{k_1}st^{k_2}s^{-1})
\in \cA.
\end{equation}
Write $t^kst^ks^{-1} = gsgs^{-1} = u^\ell$ for some primitive element $u \in \Gamma$ and
non-zero integer $\ell$. Since $g, s \in N$ and $N \le \Gamma$, we see that
\[
(gsgs^{-1})^j = u^{\ell j} \in N \qquad \forall\, j \geq 1.
\]
Applying Proposition~\ref{prop:selective_averaging_N} to $u$ and $y$, we get $E_u(y) \in \cA$. We claim that
$\E_u(y) = c_g^2\,\lambda(gsgs^{-1})$. It suffices to show that
\[
\langle u \rangle \cap \{t^{k_1}st^{k_2}s^{-1} : k_1, k_2 \in \mathbb{Z}\}
= \{e,\, gsgs^{-1}\}.
\]
If $u^i = t^{k_1}st^{k_2}s^{-1}$ for some $i \neq 0$, then since
$u^\ell = gsgs^{-1} = t^kst^ks^{-1}$ commutes with $u^i = t^{k_1}st^{k_2}s^{-1}$,
\cite[Lemma~2.3]{amrutam2025invariantC*} applied in $\langle t, s \rangle \cong F_2$
gives $k_1 = k_2 = k$. Hence $u^i = u^\ell$, and primitivity of $u$ forces $i = \ell$.
The claim follows. Therefore $\E_u(y) = c_g^2\,\lambda(gsgs^{-1}) \in \cA$, and since
$c_g \neq 0$, we get that
\begin{equation}\label{eq:gsgs}
\lambda(gsgs^{-1}) \in \cA.
\end{equation}
Replace $s$ by $s^2 \in N$ (noting that
$\langle t, s^2 \rangle \cong \langle t \rangle \star \langle s^2 \rangle \cong F_2$
since $\langle t, s \rangle \cong F_2$). The same argument gives
$\lambda(gs^2gs^{-2}) \in \cA$. Now a computation shows that
\begin{align*}
\lambda(s)\lambda(sg^{-1}s^{-1}g)\lambda(s)^{-1}
&= \lambda(s^2g^{-1}s^{-1}gs^{-1})\\
&= \lambda(s^2g^{-1}s^{-2})\cdot\lambda(sgs^{-1})\\
&= \lambda(gs^2gs^{-2})^{-1}\cdot\lambda(gsgs^{-1}) \in \cA,
\end{align*}
where we have used that $s^2g^{-1}s^{-2}g^{-1} = (gs^2gs^{-2})^{-1}$. Since $s \in N$,
$N$-invariance applied to the element $\lambda(s^2g^{-1}s^{-1}gs^{-1}) \in \cA$ gives
$\lambda(sg^{-1}s^{-1}g) \in \cA$. Hence,
\begin{equation}\label{eq:g2}
\lambda(g)^2 = \lambda(gsgs^{-1})\cdot\lambda(sg^{-1}s^{-1}g) \in \cA.
\end{equation}
Since $\lambda(g^2) \in \cA$ and $s \in N$, $N$-invariance gives
$\lambda(sg^2s^{-1}) \in \cA$. Since $\E_t(a) \in \cA$, we get that
\begin{equation}\label{eq:product2}
\E_t(a)\cdot\lambda(sg^2s^{-1})
= \sum_{m \in \mathbb{Z}} c_{t^m}\lambda(t^mst^{2k}s^{-1}) \in \cA.
\end{equation}
Write $t^kst^{2k}s^{-1} = gsg^2s^{-1} = w^{\ell'}$ for some primitive element
$w \in \Gamma$ and non-zero integer $\ell'$. Since $g, s \in N$ and
$N \trianglelefteq \Gamma$, we have $gsg^2s^{-1} \in N$, so $w^{\ell' j} \in N$ for all
$j \geq 1$. Applying Proposition~\ref{prop:selective_averaging_N} to $w$ and
$\E_t(a)\cdot\lambda(sg^2s^{-1})$, we see that
\[
\E_w\!\left(\E_t(a)\cdot\lambda(sg^2s^{-1})\right) \in \cA.
\]
Using~\cite[Lemma~2.3]{amrutam2025invariantC*} in $\langle t, s \rangle \cong F_2$, the
only element of $\langle w \rangle \cap \{t^m s t^{2k} s^{-1} : m \in \mathbb{Z}\}$ with
non-zero coefficient is $t^k st^{2k}s^{-1}$ (corresponding to $m = k$). Therefore,
\[
E_w\!\left(E_t(a)\cdot\lambda(sg^2s^{-1})\right) = c_g\,\lambda(gsg^2s^{-1}) \in \cA,
\]
and since $c_g \neq 0$:
\begin{equation}\label{eq:gsg2s}
\lambda(gsg^2s^{-1}) \in \cA.
\end{equation}
Now,
\begin{align*}
\lambda(gsg^2s^{-1})\cdot\lambda(gsgs^{-1})^{-1}
&= \lambda(t^kst^{2k}s^{-1})\cdot\lambda(t^kst^ks^{-1})^{-1}\\
&= \lambda(t^kst^{2k}s^{-1}\cdot st^{-k}s^{-1}t^{-k})\\
&= \lambda(t^k\cdot st^{2k}t^{-k}s^{-1}\cdot t^{-k})\\
&= \lambda(t^k\cdot st^ks^{-1}\cdot t^{-k})\\
&= \lambda((t^ks)\,t^k\,(t^ks)^{-1})\\
&= \lambda((gs)\,g\,(gs)^{-1}) \in \cA,
\end{align*}
where we used \eqref{eq:gsg2s} and \eqref{eq:gsgs}. Since $g \in N$ and $s \in N$, we
have $gs \in N$. Applying $N$-invariance with $(gs)^{-1} \in N$:
\[
\lambda((gs)^{-1})\cdot\lambda((gs)g(gs)^{-1})\cdot\lambda(gs) = \lambda(g) \in \cA.
\]
This completes the proof.
\end{proof}
\section{The relative ISR property for acylindrically hyperbolic
groups}
\label{sec:ISR}

We now have all the ingredients to prove the main theorems. The
argument is unified as follows. If $\cM = \bC$, there is nothing to
prove. If $\cM \neq \bC$, using
Corollary~\ref{cor:trivial_lox_support}, there exists an element
$a \in \cM$ whose Fourier support contains a loxodromic element
$n \in N$. Theorem~\ref{thm:N_invariant_position} then forces
$\lambda(n) \in \cM$, so $\cM \cap L(N) \neq \bC$. From this
foothold, the commutant rigidity of Lemma~\ref{lem:plumpness}
determines $\mathbb{E}_\cM(\lambda(g))$ for every $g \in \Gamma$ and
the conclusion $\cM = L(K)$ follows.
\begin{proof}[Proof of Theorem~\ref{thm:relativeISR}]
If $\cM = \bC$ the conclusion holds trivially. Assume therefore that
$\cM \neq \bC$. Using
Corollary~\ref{cor:trivial_lox_support}, there exists $a \in \cM$ and
a loxodromic element $n_0 \in N$ such that $\tau_0(a\lambda(n_0)^*)
\neq 0$. By Theorem~\ref{thm:N_invariant_position}, $\lambda(n_0) \in
\cM$, so $\cM \cap L(N) \neq \bC$.

Define
\[
  K_0 = \{s \in N : \lambda(s) \in \cM\}.
\]
Since $\cM$ is an $N$-invariant von Neumann subalgebra, $K_0$ is a normal subgroup of $N$. Since $\lambda(n_0) \in \cM$ and $n_0 \in N$,
$K_0$ is non-trivial. Moreover, $L(K_0) \subseteq \cM \cap L(N)$.

Let $g \in \Gamma \setminus \{e\}$ be arbitrary and set
\[
  H = K_0 \cap g^{-1}K_0 g.
\]
Since $N \trianglelefteq \Gamma$ and $K_0 \trianglelefteq N$, the
subgroup $g^{-1}K_0 g$ is also normal in $N$, so $H \trianglelefteq
N$. Since $K_0 \trianglelefteq N \trianglelefteq \Gamma$ and $\Gamma$
has trivial amenable radical, the same argument as in
Lemma~\ref{lem:plumpness} shows that $K_0$ is $s$-normal in $\Gamma$,
giving $|K_0 \cap g^{-1}K_0 g| = \infty$. In particular, $H$ is
infinite, hence non-trivial. By Lemma~\ref{lem:plumpness},
\[
  L(H)' \cap L(\Gamma) = \bC.
\]
Let $y = \lambda(g)^* \E_\cM(\lambda(g))$. We claim $y$ commutes with
every $t \in L(H)$. Since $H \subseteq K_0$ and $L(K_0) \subseteq
\cM$, we have $t \in \cM$. By the right $\cM$-module property of
$\E_\cM$,
\[
  yt = \lambda(g)^* \E_\cM(\lambda(g)t) = \lambda(g)^*\E_\cM(\lambda(g)t).
\]
Since $t \in L(H) \subseteq L(g^{-1}K_0 g)$, the element
$\lambda(g)t\lambda(g)^* \in L(K_0) \subseteq \cM$. Writing
$\lambda(g)t = (\lambda(g)t\lambda(g)^*)\lambda(g)$ and using the left
$\cM$-module property gives $yt = ty$. Because $L(H)' \cap L(\Gamma)
= \bC$, there exists $c_g \in \bC$ with
\[
  \lambda(g)^* \E_\cM(\lambda(g)) = c_g \cdot 1
  \implies \E_\cM(\lambda(g)) = c_g \lambda(g).
\]
Applying $\E_\cM$ to both sides yields $c_g = c_g^2$, so $c_g \in
\{0,1\}$. Setting $K = \{g \in \Gamma \mid \lambda(g) \in \cM\}$, it
is clear that $\cM = L(K)$.
\end{proof}

With the relative ISR-property established for von Neumann algebras, we now turn to the $C^*$-algebraic setting. To prove the analogous rigidity for reduced group $C^*$-algebras, our strategy is to pass to the weak closure and invoke Theorem~\ref{thm:relativeISR}. From there, the Fourier coefficient machinery developed in Section~\ref{sec:cstar} allows us to cleanly pull the subgroup structure back down to the $C^*$-level using $C^*$-irreducibility.

While the von Neumann algebraic
result, Theorem~\ref{thm:relativeISR}, holds in the full generality of
torsion-free acylindrically hyperbolic groups, the $C^*$-algebraic
analogue requires the more restrictive hypothesis that $\Gamma$ is a
torsion-free hyperbolic group. The reason for this restriction is that for
general torsion-free acylindrically hyperbolic group, elements of $H
\cap N$ may be elliptic. In the torsion-free hyperbolic setting,
however, every non-trivial element is loxodromic, so this difficulty
does not arise.
\begin{proof}[Proof of Theorem~\ref{thm:CstarISR}]
Assume that $\mathcal{A}\cap C_r^*(N)=\mathbb{C}$.
Let $a \in \cA$ with $\tau_0(a) = 0$. Let $\gamma \in \operatorname{Supp}(a)$, i.e.,
$c_\gamma := \tau_0(a\lambda(\gamma)^*) \neq 0$. Note that it follows from
Theorem~\ref{thm:N_invariant_position} that $\operatorname{Supp}(a) \subset \Gamma \setminus N$.
Given $\epsilon > 0$ (to be chosen later), we can choose a finite subset
$F \subset \Gamma \setminus N$ containing $\gamma$ such that the truncation
$a_0 = \sum_{s \in F} c_s \lambda(s)$ satisfies $\|a - a_0\|_2 < \epsilon$. Let
$r = a - a_0$. Let $Z = \{s\gamma^{-1} \mid s \in F \setminus \{\gamma\}\}$. Since $N$
is relatively i.c.c.\ in $\Gamma$, we can choose $n_0 \in N \setminus C_N(\gamma)$ such
that $n_0 Z n_0^{-1} \cap Z = \emptyset$. Let $b = \lambda(n_0) a \lambda(n_0)^* \in \cA$.
Moreover, with $b_0 = \lambda(n_0) a_0 \lambda(n_0)^*$, we see that $\|b\|_2 = \|a\|_2$
and $\|b - b_0\|_2 = \|r\|_2 < \epsilon$. Observe that
$g_0: = \gamma^{-1} n_0 \gamma n_0^{-1} \in N \setminus \{e\}$. We evaluate the Fourier
coefficient of $a^* b \in \cA$ at $g_0$, i.e.,
\[
|\tau_0(a^* b \lambda(g_0)^*)| = |\langle b \lambda(g_0)^*,a \rangle|
\leq |\langle b \lambda(g_0)^*,a \rangle - \langle b_0 \lambda(g_0)^*, a_0 \rangle|
+ |\langle b_0 \lambda(g_0)^*, a_0 \rangle|.
\]
Expanding the inner product and applying the Cauchy--Schwarz and triangle inequalities
(noting $\|a_0\|_2 = \|a - r\|_2 \leq \|a\|_2 + \|r\|_2 < \|a\|_2 + \epsilon$), we get
that
\begin{align*}
|\langle b \lambda(g_0)^*, a\rangle - \langle b_0 \lambda(g_0)^*,a_0 \rangle|
&\leq |\langle b \lambda(g_0)^*,r \rangle| + |\langle (b - b_0) \lambda(g_0)^*,a_0 \rangle| \\
&\leq \|r\|_2 \|b\|_2 + \|b - b_0\|_2 \|a_0\|_2 \\
&< \epsilon \|a\|_2 + \epsilon (\|a\|_2 + \epsilon) = 2\epsilon\|a\|_2 + \epsilon^2.
\end{align*}
To compute the principal term $\langle b_0 \lambda(g_0)^*, a_0 \rangle$, we expand the
finite sums to get
\[
\langle b_0 \lambda(g_0)^*, a_0 \rangle = \tau_0(a_0^* b_0 \lambda(g_0)^*)
= \sum_{s, t \in F} \overline{c_s}\, c_t\,
\tau_0(\lambda(s^{-1} n_0 t n_0^{-1} g_0^{-1})).
\]
The trace is non-zero only if $s^{-1} n_0 t n_0^{-1} = g_0 = \gamma^{-1} n_0 \gamma n_0^{-1}$.
Rearranging gives $n_0(t\gamma^{-1})n_0^{-1} = s\gamma^{-1}$. By our choice of $n_0$ and
the disjointness $n_0 Z n_0^{-1} \cap Z = \emptyset$ (noting $e \notin Z$), this occurs
if and only if $s = t = \gamma$. Thus, the sum collapses to
$\langle b_0 \lambda(g_0)^*, a_0 \rangle = |c_\gamma|^2$. We now apply the reverse
triangle inequality to get
\[
|\langle b \lambda(g_0)^*,a \rangle| \geq |\langle b_0 \lambda(g_0)^*, a_0 \rangle|
- |\langle b \lambda(g_0)^*, a\rangle - \langle b_0 \lambda(g_0)^*,a_0 \rangle|.
\]
Substituting our computed principal term and the upper bound for the error term, we obtain
\[
|\tau_0(a^* b \lambda(g_0)^*)| \geq |c_\gamma|^2 - (2\epsilon\|a\|_2 + \epsilon^2).
\]
Choosing $\epsilon > 0$ small enough such that $2\epsilon\|a\|_2 + \epsilon^2 < |c_\gamma|^2$
ensures the Fourier coefficient is strictly positive. By
Theorem~\ref{thm:N_invariant_position}, $\lambda(g_0) \in \cA$, which contradicts
$\cA \cap C_r^*(N) = \mathbb{C}$. Therefore, $\mathcal{A}=\mathbb{C}$. 

Suppose now that $\cA \cap C_r^*(N) \neq \mathbb{C}$. Then, $\mathcal{A}_N:=\mathcal{A}\cap C_r^*(N)$
is a $N$-invariant $C^*$-subalgebra of $C_r^*(N)$. It then follows from Theorem~\ref{thm:N_invariant_position} that if there exists $a\in\mathcal{A}_N$ with $\tau_0(a\lambda(g)^*)\ne 0$ for some $g\in N$, then $\lambda(g)\in\mathcal{A}_N$. Letting $K=\{g\in N: \lambda(g)\in \mathcal{A}_N\}$, it follows from \cite[Proposition~2.1]{amrutam2025invariantC*} that $\mathcal{A}_N=C_r^*(K)$ for some normal subgroup $K_0\trianglelefteq N$. 
Since $\cA \subseteq C_r^*(\Gamma)$ and $\cA'' = L(H)$ (by Theorem~\ref{thm:relativeISR}), we have that $$\cA \subseteq \cA'' \cap C_r^*(\Gamma) = L(H) \cap C_r^*(\Gamma).$$It is a standard fact that $L(H) \cap C_r^*(\Gamma) = C_r^*(H)$, which implies that $\cA \subseteq C_r^*(H)$. Next, we claim that $K = H \cap N$. Because $\cA \subseteq C_r^*(H)$, we have that$$C_r^*(K) = \cA \cap C_r^*(N) \subseteq C_r^*(H) \cap C_r^*(N) = C_r^*(H \cap N).$$ This immediately implies the subgroup inclusion $K \leq H \cap N$. To prove the reverse inclusion, let $h \in H \cap N$. Because $\cA$ is weakly dense in $L(H)$ and $\lambda(h) \in L(H)$, there must exist an element $a \in \cA$ such that its Fourier coefficient at $h$ is non-zero. Indeed, if $\tau_0(a\lambda(h)^*) = 0$ for all $a \in \cA$, weak continuity of the inner product would imply $\tau_0(x\lambda(h)^*) = 0$ for all $x \in L(H)$, which is absurd for $x = \lambda(h)$. Since $h \in N$ and we have found $a \in \cA$ with $c_h(a) \neq 0$, Theorem~\ref{thm:N_invariant_position} guarantees that $\lambda(h) \in \cA$. Consequently, $\lambda(h) \in \cA \cap C_r^*(N) = C_r^*(K)$, which forces $h \in K$. Therefore, $K = H \cap N$. Finally, we show that $H$ normalizes $K$ (i.e., $H \subseteq N_\Gamma(K)$). Let $h \in H$ and $k \in K$. We consider the conjugate $h k h^{-1}$. Since $k \in K \subseteq H$ and $H$ is a subgroup, closure under conjugation gives $h k h^{-1} \in H$. Moreover, since $k \in K \subseteq N$ and $N \trianglelefteq \Gamma$, it follows that $h k h^{-1} \in N$. Combining these facts, $h k h^{-1} \in H \cap N = K$. Since this holds for all $h \in H$, we have $h K h^{-1} = K$, which proves that $H$ normalizes $K$. Hence, $H \subseteq N_\Gamma(K)$. Since we established that $\cA \subseteq C_r^*(H)$, we conclude that $$C_r^*(K)\subset\cA \subseteq C_r^*(H) \subseteq C_r^*(N_\Gamma(K)).$$
Since $K \trianglelefteq N \trianglelefteq \Gamma$ and $\Gamma$ has trivial amenable
radical, hence $C^*$-simple (see~\cite{AD19}), and $C^*$-simplicity passes to normal subgroups (see~\cite[Theorem~1.4]{breuillard2017c}), it follows that $C_r^*(K)$ is simple. Moreover, note that $C_r^*(K)'\cap C_r^*(N_{\Gamma}(K))=\mathbb{C}$ (using Lemma~\ref{lem:plumpness}). In particular, using~\cite[Theorem~1.3]{ursu2022relative}, we see that $K$ is plump in $N_{\Gamma}(K)$ (in the sense of \cite{amrutam2021intermediate}). Therefore, every intermediate $C^*$-algebra $\mathcal{B}$ with $C_r^*(K)\subseteq\mathcal{B}\subseteq C_r^*(N_{\Gamma}(K))$ is simple, in particular, the inclusion $C_r^*(K)\subset C_r^*(N_{\Gamma}(K))$ is $C^*$-irreducible in the sense of \cite{rordam2023irreducible}. We can now appeal to \cite[Theorem~5.3]{bedos2023c} to conclude that $\mathcal{A}=C_r^*(K_0)$ for some subgroup $K_0\le\Gamma$.  
\end{proof}
\begin{remark}
In \cite[Subsection~4.3]{amrutam2025invariantC*}, it was remarked that following the arguments made for the torsion-free hyperbolic group, all torsion-free acylindrically hyperbolic groups with trivial amenable radical have $C^*$-ISR-property. While the conclusion is correct, it does not immediately follow from the torsion-free hyperbolic group case, owing to the existence of elliptic elements in the case of torsion-free acylindrically hyperbolic groups. The correct argument is as follows. Let $\mathcal{A}$ be a $\Gamma$-invariant unital $C^*$-subalgebra of $C_r^*(\Gamma)$, where $\Gamma$ is a torsion-free acylindrically hyperbolic group with trivial amenable radical. If $\mathcal{A}\ne \mathbb{C}$, it follow from Corollary~\ref{cor:cstar_trivial_lox_support} that there exists a loxodromic $g\in \Gamma$ and $a\in\mathcal{A}$ such that $\tau_0(a\lambda(g)^*)\ne 0$. Now, arguing similarly as in Theorem~\ref{thm:N_invariant_position}, we will obtain that $\lambda(g)\in\mathcal{A}$. Now, we can argue as in the last part of Theorem~\ref{thm:CstarISR} to get that $C_r^*(N)\subseteq\mathcal{A}\subseteq C_r^*(\Gamma)$, where $N=\{g\in\Gamma: \lambda(g)\in\cA\}$ with the inclusion $C_r^*(N)\subset C_r^*(\Gamma)$ irreducible. It now follows from \cite[Theorem~5.3]{bedos2023c} to get that $\mathcal{A}$ must come from a normal subgroup.  
\end{remark}
\section{Relative ISR for Higher-Rank Lattices}\label{sec:lattice}

In this section, we establish Theorem~\ref{thm:latticeISR}, proving the relative ISR-property for irreducible lattices in higher-rank semisimple Lie groups. The setting completely distances itself from acylindrical hyperbolicity and instead leverages unitary factorization along with the Character Decomposition Property. 

Our strategy is
heavily influenced by the approach developed by Dudko and Jiang~\cite{dudko2024character},
where the Character Decomposition Property (CDP) -- introduced and systematically studied
for the first time -- plays a central role. The argument proceeds in
three steps. 

First, a structural unitary factorization lemma (Lemma~\ref{lem:factorisation})
decomposes unitaries $\lambda(g)$ with non-zero conditional expectation into $\cM$ as a
product of a unitary in $\cM$ and a unitary in its relative commutant
$\cM' \cap L(\Gamma)$. Second, a direct computation shows that the positive-definite
functions $\phi_\cM$ and $\phi_{\cM'}$ associated to $\cM$ and its relative commutant
satisfy $\phi_\cM(g)\phi_{\cM'}(g) = 0$ for all $g \in N \setminus \{e\}$. Third, the
CDP inherited by $N$ from $\Gamma$ -- via the Margulis Normal Subgroup Theorem and the
inheritance result of~\cite[Proposition~3.2(1)]{dudko2024character} -- forces one of
these functions to be identically $\delta_e$, from which the conclusion $\cM = L(K)$ is
extracted. 
\subsection{Structural Lemmas on Unitary Factorization}
    We begin by recording the following algebraic lemma for the sake of completion. This has already been established in \cite{chifan2020rigidity} and \cite{chifan2023invariant}.
\begin{lemma}[Unitary factorisation]\label{lem:factorisation}
Let $\mathcal{P} \subseteq \mathcal{Q}$ be an inclusion of von Neumann algebras with $(\mathcal{Q}, \tau)$ a tracial von Neumann factor. Suppose $u \in \mathcal{Q}$ is a unitary such that $u\mathcal{P} u^* \subseteq \mathcal{P}$. Assume further that the center of the subalgebra is trivial, $\mathcal{Z}(\mathcal{P}) = \bC$. If $\E_{\mathcal{P}}(u) \neq 0$, then $u$ factors as
\[
u = a(u)b(u)
\]
where $a(u) \in \mathcal{U}(\mathcal{P})$ and $b(u) \in \mathcal{U}(\mathcal{P}' \cap \mathcal{Q})$. Explicitly, $a(u) = c_u^{-1/2}\E_{\mathcal{P}}(u)$ for some $c_u \in \bR_{>0}$.
\end{lemma}
\begin{proof}
    Let $\tilde{b}(u) = u^*\E_{\mathcal{P}}(u)$. By arguing as in \cite[Theorem~3.15]{chifan2020rigidity}, we see that $\tilde{b}(u) \in \mathcal{P}' \cap \mathcal{Q}$. Observe that
\[
  \E_{\mathcal{P}}(u)^*\E_{\mathcal{P}}(u)
  = \E_{\mathcal{P}}(u)^*\,u\,u^*\,\E_{\mathcal{P}}(u)
  = \bigl(u^*\E_{\mathcal{P}}(u)\bigr)^*\bigl(u^*\E_{\mathcal{P}}(u)\bigr)
  = \tilde{b}(u)^*\tilde{b}(u).
\]
Because $\tilde{b}(u) \in \mathcal{P}' \cap \mathcal{Q}$, the product $\tilde{b}(u)^*\tilde{b}(u)$ also lies in $\mathcal{P}' \cap \mathcal{Q}$. However, $\E_{\mathcal{P}}(u)^*\E_{\mathcal{P}}(u) \in \mathcal{P}$. Consequently,
\[
  \E_{\mathcal{P}}(u)^*\E_{\mathcal{P}}(u) \in \mathcal{P} \cap (\mathcal{P}' \cap \mathcal{Q}) = \mathcal{Z}(\mathcal{P}) = \bC.
\]
Since it is a non-zero positive element, $\E_{\mathcal{P}}(u)^*\E_{\mathcal{P}}(u) = c_u$ for some constant $c_u > 0$ (observe that $c_u=c_u^*$). Moreover, since $\cM$ is tracial, applying $\tau$ on both sides, we obtain that 
$$c_u=\tau(c_u)=\tau\left(\mathbb{E}_{\mathcal{P}}(u)^*\mathbb{E}_{\mathcal{P}}(u)\right)=\tau\left(\mathbb{E}_{\mathcal{P}}(u)\mathbb{E}_{\mathcal{P}}(u)^*\right)=\tau(c_{u^*})=c_{u^*}$$
We now isolate $u$ by substituting $u = \E_{\mathcal{P}}(u)\tilde{b}(u)^*$ to obtain
\[
  c_u\, u = \E_{\mathcal{P}}(u)\E_{\mathcal{P}}(u)^* u = \E_{\mathcal{P}}(u)\tilde{b}(u)^*.
\]
Dividing by $c_u$, we obtain the factorisation $u = a(u)b(u)$, where
\[
  a(u) = \frac{\E_{\mathcal{P}}(u)}{\sqrt{c_u}}, \qquad b(u) = \frac{\tilde{b}(u)^*}{\sqrt{c_u}}.
\]
By construction, $a(u) = w \in \mathcal{U}(\mathcal{P})$. Furthermore, since $u$ and $a(u)$ are unitaries, $b(u) = a(u)^*u$ must be unitary, and it lies in $\mathcal{P}' \cap \mathcal{Q}$ since $\tilde{b}(u) \in \mathcal{P}' \cap \mathcal{Q}$.
\end{proof}

\subsection{Proof of Theorem~\ref{thm:latticeISR}}
We are now prepared to complete the proof of Theorem~\ref{thm:latticeISR}. The argument relies on checking the behavior of conditional expectations across normal subgroups, followed by a direct application of the Character Decomposition Property.

\begin{proof}[Proof of Theorem~\ref{thm:latticeISR}]
Let $\Gamma = \mathrm{SL}_d(\bZ)$ where $d \geq 3$ odd, and let
$N \trianglelefteq \Gamma$ be a non-trivial normal subgroup. Since $\Gamma$ is
an irreducible lattice in a higher-rank semisimple Lie group with trivial center
and no compact factors, the Margulis Normal Subgroup Theorem
(Theorem~\ref{thm:margulis}) implies that $N$ has finite index in $\Gamma$. We
henceforth work with this finite-index normal subgroup $N$. We are given an
$N$-invariant von Neumann subalgebra $\cM \leq L(\Gamma)$. It follows from
Proposition~\ref{prop:amenable_subalgebra} that $\mathcal{Z}(\cM) = \bC$.
\textit{Claim-1:} If $\E_{\cM}(\lambda(n)) = 0$ for all $n \in N
\setminus \{e\}$, then $\cM = \bC$.

We show that $N \setminus \{e\}$ is $(\cM, N)$-thick, so that Lemma~\ref{lem:GNS_vanishing} applies directly. Fix
$g \in \Gamma \setminus \{e\}$. Let $C_N(g) = C_\Gamma(g) \cap N$.
Since $\Gamma = \mathrm{SL}_d(\bZ)$ is an ICC group, $C_\Gamma(g)$
has infinite index in $\Gamma$. By the Margulis NST
(Theorem~\ref{thm:margulis}), $[\Gamma : N] < \infty$, forcing
$[N : C_N(g)] = \infty$. Choose an infinite sequence
$\{n_i\}_{i \geq 1} \subset N$ from distinct left cosets of $C_N(g)$.
For $i \neq j$, the element
\[
  w_{i,j} = n_j^{-1}\,g\,n_j \cdot n_i^{-1}\,g^{-1}\,n_i \in
  N \setminus \{e\},
\]
since $n_i$ and $n_j$ lie in distinct cosets of $C_N(g)$. Hence $N
\setminus \{e\}$ is $(\cM, N)$-thick. Since $\mathbb{E}_\cM(\lambda(n))
= 0$ for all $n \in N \setminus \{e\}$ by hypothesis, Lemma~\ref{lem:GNS_vanishing} gives $\cM = \bC$, finishing the proof of \textit{Claim-1}.

Now, by an abuse of notation, let us denote the relative commutant $\cM'\cap L(\Gamma)$ by just $\cM'$. Observe that $\cM'$ is also $N$-invariant and hence, $$\phi_{\cM'}(g) = \tau_0(\mathbb{E}_{\cM'}(\lambda(g))\lambda(g)^*)=\|\E_{\cM'}(\lambda(g))\|_2^2$$ is also an $N$-invariant positive definite function on the group $\Gamma$.\\  
\textit{Claim-2:} $\phi_{\cM}(g)\phi_{\cM'}(g)=\delta_{e,g}$ for $g\in N$.\\
Pick $g \in N \setminus \{e\}$. If $\E_{\cM}(\lambda(g)) = 0$, the product $\phi_{\cM}(g)\phi_{\cM'}(g) = 0$ just by definition. Now, suppose that $\E_{\cM}(\lambda(g)) \ne 0$. Applying Lemma~\ref{lem:factorisation} to the inclusion $\cM \subseteq L(\Gamma)$, we can write $\lambda(g) = a(g)\,b(g)$ with
$a(g) \in \mathcal{U}(\cM)$ (unitary in $\cM$) and $b(g) \in
\mathcal{U}(\cM' \cap L(\Gamma))$ (unitary in $\cM'$).
We observe that $\E_{\cM'}(a(g))= \tau(a(g))$. Indeed, take any $x \in \cM'$. Since $a(g) \in \cM$ and $x \in \cM'$, we have
\[
  \E_{\cM'}(a(g))\,x
  = \E_{\cM'}(a(g)\,x)
  = \E_{\cM'}(x\,a(g))
  = x\,\E_{\cM'}(a(g)).
\]
Thus, $\E_{\cM'}(a(g))$ commutes with every $x \in \cM'$, so that using Proposition~\ref{prop:amenable_subalgebra}, we get that
$$\E_{\cM'}(a(g)) \in (\cM'\cap L(\Gamma))' \cap (\cM'\cap L(\Gamma))=\mathcal{Z}(\cM'\cap L(\Gamma))=\mathbb{C}.$$
Hence
$\E_{\cM'}(a(g)) = \tau(\E_{\cM'}(a(g))) = \tau(a(g)) $. A similar argument shows that
$\E_\cM(b(g)) = \tau(b(g))$.

Let us now compute $\phi_\cM(g)$. Since $a(g) \in \cM$ and $b(g) \in \cM'$, we see that
\[
  \E_\cM(\lambda(g))
  = \E_\cM(a(g)\,b(g))
  = a(g)\,\E_\cM(b(g))
  = a(g)\cdot\tau(b(g)).
\]
Since $a(g)$ and $b(g)$ are unitaries, we observe that
\begin{align*}
  \phi_\cM(g)
  &= \tau\bigl(\E_\cM(\lambda(g))\,\lambda(g)^{-1}\bigr)
  \\&= \tau\bigl(a(g)\,\tau(b(g))\cdot b(g)^*\,a(g)^*\bigr) \\
  &= \tau(b(g))\cdot\tau\bigl(a(g)\,b(g)^*\,a(g)^*\bigr)\\& = \tau(b(g))\cdot\tau\bigl(b(g)^*\,a(g)^*\,a(g)\bigr)
     \\& = \tau(b(g))\cdot\tau(b(g)^*)
     = \tau(b(g))\cdot\overline{\tau(b(g))} = |\tau(b(g))|^2.
\end{align*}
A similar computation shows that
$\phi_{\cM'}(g) = |\tau(a(g))|^2$. Now,
\begin{align*}
  \phi_\cM(g)\cdot\phi_{\cM'}(g)
  &= |\tau(b(g))|^2\cdot|\tau(a(g))|^2
     \\
  &= |\tau(b(g))\cdot\tau(a(g))|^2 \\
  &= |\tau\bigl(b(g)\cdot\tau(a(g))\bigr)|^2
     \\
  &= |\tau\bigl(\E_{\cM'}(a(g))\cdot b(g)\bigr)|^2
     \\
  &= |\tau\bigl(\E_{\cM'}(a(g)\,b(g))\bigr)|^2
     && (b(g)\in\cM') \\
  &= |\tau\bigl(\E_{\cM'}(\lambda(g))\bigr)|^2
     && (\lambda(g)=a(g)b(g)) \\
  &= |\tau(\lambda(g))|^2
     && (\tau\circ\E_{\cM'}=\tau) \\
  &= 0.
\end{align*}
Assume now that $\cM\ne\mathbb{C}$. It follows from \textit{Claim-1} that there exists $g\in N\setminus\{e\}$ such that $\mathbb{E}_\cM(\lambda(g))\ne 0$ which is equivalent to saying that $\phi_{\cM}(g)\ne 0$. Since $N$ inherits the CDP from $\Gamma$
by~\cite[Proposition~3.2(1)]{dudko2024character}, and $\phi_{\cM}|_N$,
$\phi_{\cM'}|_N$ are characters on $N$ satisfying
$\phi_{\cM}(g)\phi_{\cM'}(g) = 0$ for all $g \in N \setminus \{e\}$ by
\textit{Claim~2}, the CDP gives either $\phi_{\cM}|_N \equiv \delta_e$ or
$\phi_{\cM'}|_N \equiv \delta_e$.
Since $\E_{\cM}(\lambda(g)) \neq 0$, $\phi_{\cM} \not\equiv \delta_e$. Therefore, we must have $\phi_{\cM'} \equiv \delta_e$, meaning $\E_{\cM' \cap L(\Gamma)}(\lambda(g)) = 0$ for all $g \in N \setminus \{e\}$. Applying \textit{Claim-1} to the $N$-invariant subalgebra $\mathcal{P} = \cM' \cap L(\Gamma)$, we obtain $\cM' \cap L(\Gamma) = \bC$.

Let $K = \{g \in N \mid \E_\cM(\lambda(g)) \neq 0\}$. For any $g \in K$, we factored $\lambda(g) = a(g)b(g)$ with $a(g) \in \mathcal{U}(\cM)$ and $b(g) \in \mathcal{U}(\cM' \cap L(\Gamma))$. Since we established that $\cM' \cap L(\Gamma) = \bC$, $b(g)$ must be a scalar of modulus $1$. Consequently, $\lambda(g) = a(g)b(g) \in \cM$. This proves that $K = \{g \in N \mid \lambda(g) \in \cM\}$, which is clearly a subgroup of $N$. Moreover, since $\cM$ is $N$-invariant, for any $n \in N$ and $k \in K$, we have $\lambda(nkn^{-1}) = \lambda(n)\lambda(k)\lambda(n)^* \in \cM$, meaning $nkn^{-1} \in K$. Hence, $K$ is a normal subgroup of $N$. Since $\phi_\cM \not\equiv \delta_e$, $K$ is non-trivial. Moreover, $L(K)\subset\cM$.

Because $N$ is itself an irreducible lattice in a higher-rank semisimple Lie group, the Margulis Normal Subgroup Theorem implies that its non-trivial normal subgroup $K$ must have finite index in $N$. Since $N$ has finite index in $\Gamma$, $K$ has finite index in $\Gamma$. Consequently, its normal core $K_0 = \operatorname{Core}_\Gamma(K) = \bigcap_{\gamma \in \Gamma} \gamma K \gamma^{-1}$ is a finite intersection of finite-index subgroups, which implies $K_0$ is a finite-index normal subgroup of $\Gamma$. In particular, $K_0$ is non-trivial. 

Hence, $K_0$ is Plump in $\Gamma$ by \cite[Corollary~3.4]{amrutam2021intermediate}. Consequently, we have $L(K_0)\subseteq \cM\subseteq L(\Gamma)$ with $L(K_0)'\cap L(\Gamma)=\mathbb{C}$. We can now appeal to \cite[Proposition~4.4]{bedos2023c} to obtain that $\mathcal{M}=L(H)$ for some subgroup $H\le \Gamma$.
\end{proof}
The relative $C^*$-ISR property need not necessarily extend to their direct products as the following example demonstrates.
\begin{example}\label{ex:product_counterexample}
Let $\Gamma = \mathbb{F}_n \times \mathbb{F}_m$ (for $n, m \geq 2$). Let
$N = \mathbb{F}_n \times \{e\} \trianglelefteq \Gamma$. Let
$a_1, a_2, \ldots, a_m$ denote the standard free generators of $\mathbb{F}_m$.
Define an automorphism $\sigma \in \operatorname{Aut}(\mathbb{F}_m)$ by cyclic
permutation of the generators: $\sigma(a_i) = a_{i+1 \pmod{m}}$. This induces
a canonical trace-preserving automorphism on $C_r^*(\mathbb{F}_m)$. Let
$\mathcal{B} = (C_r^*(\mathbb{F}_m))^\sigma$ denote the fixed-point
$C^*$-subalgebra, and set
\[
\mathcal{A} = C_r^*(\mathbb{F}_n) \otimes_{\min} \mathcal{B}.
\]
Clearly $\mathcal{A}$ is $N$-invariant. We claim that $\mathcal{A}$ is not of
the form $C_r^*(H)$ for any subgroup $H \leq \Gamma$.

First, $\sigma$ has no non-trivial fixed points in $\mathbb{F}_m$. Indeed,
suppose $w \in \mathbb{F}_m$ is a reduced word satisfying $\sigma(w) = w$.
Since $\sigma$ cyclically permutes every generator and its inverse, the
permuted word $\sigma(w)$ is again reduced and equals $w$ as reduced words.
Comparing letter by letter forces each letter of $w$ to be fixed by $\sigma$.
But no individual generator or its inverse is fixed by the cyclic permutation
$a_i \mapsto a_{i+1}$ (for $m \geq 2$), so $w = e$. Consequently, for any
$g_2 \in \mathbb{F}_m$, $\lambda(g_2) \in \mathcal{B}$ if and only if
$g_2 = e$.

Now suppose for contradiction that $\mathcal{A} = C_r^*(H)$ for some
$H \leq \Gamma$. Since $\lambda(g_2) \in C_r^*(H)$ implies $g_2 \in H$, and
no non-trivial $\lambda(g_2)$ (with $g_2 \in \mathbb{F}_m$) belongs to
$\mathcal{B}$, we conclude $H \leq \mathbb{F}_n \times \{e\}$, so that
$C_r^*(H) \subseteq C_r^*(\mathbb{F}_n) \otimes \mathbb{C}$. However,
$\mathcal{B} \supsetneq \mathbb{C}$, as the element
\[
b_0 := \frac{1}{m}\sum_{i=1}^{m} \lambda(a_i) \in \mathcal{B}
\]
is non-scalar (as its Fourier support is non-trivial), so
$\mathcal{A} = C_r^*(\mathbb{F}_n) \otimes_{\min} \mathcal{B}
\not\subseteq C_r^*(\mathbb{F}_n) \otimes \mathbb{C}$, a contradiction.
\end{example}
The above example contrasts with the
ISR property for this class of groups~\cite{amrutam2023invariant,amrutam2025invariantC*}, and suggests that the relative version is
genuinely more sensitive to the global structure of the group. However, under appropriate conditions on the normal subgroups, it is plausible that $N$-invariant subalgebras do come from normal subgroups.
\bibliography{name}
\end{document}